\documentclass[12pt]{article}

\usepackage{amsmath}
\usepackage{amssymb,amsbsy,amsmath,amsfonts,amssymb,amscd,amsthm}
\usepackage{stmaryrd}
\usepackage{mathrsfs}
\usepackage[active]{srcltx}
\usepackage{bbm}
\usepackage{paralist}
\usepackage{graphics} 
\usepackage{epsfig} 
\usepackage{epstopdf} 
\usepackage[colorlinks=true]{hyperref}
\hypersetup{urlcolor=blue, citecolor=red}

\def\a{\alpha}

\def\o{\omega}
\newcommand{\etal}{\textit{et al.~}}

\def\cA{{\cal A}}

\def\cH{{\cal H}}

\def\Re {\mathcal {R}e}

\def\ii{{\mbox{i}}}

\newtheorem{theorem}{Theorem}[section]
\newtheorem{lemma}{Lemma}[section]

\newtheorem{definition}{Definition}[section]

\begin{document}
\title{Sharp Stability of a String with Local Degenerate Kelvin-Voigt Damping
\thanks{
This work was supported by the National Natural Science Foundation of China (grants No. 62073236, 61873036) and Beijing Municipal Natural Science Foundation (grant No. 4182059).
\medskip} }

\author{Zhong-Jie Han$^1$,\quad Zhuangyi Liu$^2$,\quad Qiong Zhang$^3$\thanks{Corresponding author, email: zhangqiong@bit.edu.cn}\\
	$^{1}$School of Mathematics, Tianjin University\\
	Tianjin 300354,  P. R. China
	\\
	$^2$Department of Mathematics and Statistics, University of
	Minnesota\\
	Duluth, MN 55812-3000, USA \\
	$^{3}$School of Mathematics and Statistics, Beijing Key \\Laboratory
on MCAACI, Beijing Institute of Technology, \\Beijing, 100081,   P. R. China
}

\maketitle
\begin{abstract}
	This paper is on the asymptotic behavior of the elastic string equation
	with localized degenerate Kelvin-Voigt damping
	$$
	u_{tt}(x,t)-[u_{x}(x,t)+b(x)u_{x,t}(x,t)]_{x}=0,\; x\in(-1,1),\; t>0,
	$$
	where $b(x)=0$ on $x\in (-1,0]$, and $b(x)=x^\alpha>0$ on $x\in (0,1)$ for $\alpha\in(0,1)$. It is known that the optimal decay
	rate of solution is $t^{-2}$ in the limit case $\a=0$ \cite{ARSVG}, and exponential decay rate for $\a\ge 1$ \cite{LZ}. When $\a\in (0,1)$, the damping coefficient $b(x)$ is continuous, but its derivative has a singularity at the interface $x=0$. In this case, the best known decay rate is $t^{-\frac{3-\a}{2(1-\a)}}$ \cite{HLW}. Although this rate is consistent with the exponential one at $\a=1$, it failed to match the optimal one at $\a=0$.

    In this paper, we obtain a sharper polynomial decay rate $t^{-\frac{2-\a}{1-\a}}$. More significantly, it is consistent with the optimal polynomial decay rate at $\a=0$ and the exponential decay rate at $\a = 1$.
	This is a big step toward the goal of obtaining eventually the optimal decay rate.


\end{abstract}


\section{Introduction}
\setcounter{equation}{0} \setcounter{theorem}{0}
\setcounter{lemma}{0} \setcounter{definition}{0}
\setcounter{proposition}{0} \setcounter{remark}{0}
\setcounter{corollary}{0} \setcounter{equation}{0}

In this paper, we consider elastic string (one-dimensional wave equation) with local viscoelastic damping of Kelvin-Voigt type.  The mathematical model
is the following partial differential equation.
\begin{equation}
\label{system} \left\{
\begin{array}{lcl}
 u_{tt}(x,t) -  [u_x(x,t)+    b(x) u_{xt}(x,t)]_x = 0
 & \mbox{ in } & (-1,1)\times \mathbb{R}^+, \\ \noalign{\medskip}
  u(-1, t) = u(1, t)  =0 & \mbox{ in } & \mathbb{R}^+,   \\ \noalign{\medskip}
   u(x,0)  = u_0(x) , \quad u_t(x,0)= v_0(x) & \mbox{ in } & [0,L].
\end{array}
\right.
\end{equation}
 The damping coefficient function $b(\cdot)\,:\,(-1,\,1) \to {\mathbb R}^+\cup\{0\}$, belongs to $ L^\infty(-1,1)$  and satisfies
\begin{equation}\label{h1}
b(x) =  0  \;  \;  \mbox{ for  }  \;   x\in  [-1,\,0)    \;  \;  \mbox{ and  }  \;\;
b(x) = a(x) = x^\a    \;    \mbox{with}  \; \a\in [0, 1)  \mbox{ for  } \;\;  x\in  [0,\, 1].\;
 \tag{H1}
\end{equation}

In the theory of elasticity, Kelvin-Voigt damping is a type of viscoelastic damping which assumes that the stress is a linear function of strain and strain rate. When it is globally distributed, the solution to the corresponding elastic equation (string, beam, plate) is exponentially stable and analytic \cite{Huang}. In 1998,  Chen \etal (\cite{CLL, LL1}) discovered that the semigroup associated with the above system (\ref{system}) is not exponentially stable if the damping is locally distributed and $a(x)$ is proportional to the characteristic function of any subinterval of the domain.
This surprising result revealed that, unlike the viscous damping, the Kelvin-Voigt damping does not follow the well-known Geometric Optics Condition (\cite{BLR}).
In 2002, it was proved in \cite{LL2} that system (\ref{system}) is exponentially stable if $b'(\cdot) \in C^{0,1}[-1,1]$. Later, the smoothness condition on $b(\cdot)$ was weakened to $b(\cdot) \in C^1[-1,1]$ and the following conditions (\cite{Zhang1}):
\begin{equation}\label{qzhang}
a'(0)=0,\;  \int_0^x {{|a'(s)|^2}\over{a(s)}} \: ds  \le C |a'(x)|,\; \forall x\in [0,1], \; C>0.
\end{equation}
It is easy to see that $a(x)=x^\alpha, \;\a>1$ satisfies condition (\ref{qzhang}), but
not when $\a=1$. The exponential stability of this case was confirmed later in \cite{LZ}.
 On the other hand, Liu and Rao in 2005  \cite{ZLiuRao} proved that the semigroup corresponding to system (\ref{system}) is polynomially stable of order   at  almost $2$
if $a(\cdot) \in C(0,1)$ and $a(x)\ge \tilde{a}>0 $ on $(0,1)$, which includes the case $\a=0$ in (\ref{system}).
Actually, the order can be improved to $2$ without any modification to the proof, thanks to the necessary and sufficient conditions for polynomial stability by  Borichev and  Tomilov  \cite{BT}.
More recently, the optimality of this order was confirmed in \cite{ARSVG}.
 The same  optimal polynomial  decay rate was obtained in \cite{GNW} for the damping Kelvin--Voigt mechanism acting in  any  internal of the one-dimensional domain. We also refer to \cite{hassin}, \cite{liu-zhang-2018}, \cite{aliweb}, \cite{liu-2005}, and the references therein  for other kinds of elastic models with local Kelvin-Voigt damping.

It is reasonable to expect that the system (\ref{system}) is polynomially stable whose order depends on $\a$. To our knowledge, the best known decay rate in this case is $t^{-\frac{3-\a}{2(1-\a)}}$ \cite{HLW}. Although this rate is consistent with the exponential one at $\a=1$, it failed to match the optimal one at $\a=0$ (that is, $t^{-2}$).

For the stability of higher dimensional wave equations with local Kelvin-Voigt damping, we refer the readers to the papers  \cite{AHR},  \cite{Burq1}, \cite{BC},
\cite{KLiuRao}, \cite{zhang4}, \cite{Stahn}, \cite{hanyu},  \cite{zhang3},  \cite{zhang2} and the reference therein.

Before stating the main results, we introduce the notions of stability that we encounter in this work.
\begin{definition}
Let $A\,:\,D(A) \subset H \to H$ generate a $C_0$-semigroup $e^{tA}$   on   Hilbert space  $H$.
  The semigroup $e^{tA}$ is said to be polynomially stable of order $\gamma>0$ if
 there exists a positive constant $C $  such that
 \begin{equation}
 \label{rao2}
 \|e^{At}x_0\|_{H} \leqslant C t^{-\gamma}\|Ax_0\|_H , \quad  \forall \;t\geqslant  1, \; x_0\in {\cal D}(A).
 \end{equation}
\end{definition}

In this paper, we shall give a sharper decay rate for system (\ref{system}). In fact, we obtain  a sharper polynomial decay rate $t^{-\frac{2-\a}{1-\a}}$. More significantly, this rate is consistent with the optimal polynomial decay rate at $\a=0$ and the exponential decay rate at $\a = 1$. However,
whether this rate is optimal is still an open question.  We
would like to point out here that showing optimal polynomial decay
rate often relies on some knowledge about the spectrum of the
system. However, the spectral analysis for the asymptotic behavior
of the eigenvalues of system (\ref{system}) is a formidable task
due to the degeneracy of the highest order term in its eigensystem.
In fact, although there have been a few results on the
asymptotic behavior of the spectrum for the system with global
Kelvin-Voigt damping (see \cite{GWZ}, \cite{XM}), the asymptotic
expressions of the eigenvalues for the system with locally
Kelvin-Voigt damping remains as an difficult open problem as commented
in \cite{XM}.

This paper is organized as follows. First, we present our main results and some preliminaries in Section 2. Section 3 is devoted to the proofs of the main results.

\section{Main Results and Preliminaries}
\setcounter{equation}{0} \setcounter{theorem}{0}
\setcounter{lemma}{0} \setcounter{definition}{0}
In this section, we shall recall some results concerning the well-posedness, weighted Hardy inequality and present main results.

Let $H^1_{0}(-1,1)$ be the space
 $\{u\in H^1(-1,1)\: |$ $u(-1) =u(1)= 0\}.$
We introduce a Hilbert space
$$
 {\cal{H}} =
   H^1_0(-1,\,1) \times L^2(-1,\,1),
$$
whose inner product induced norm is given by
$$
 \|U\|_{\cal{H}}  =   \sqrt{\|u\|_{H^1_0(-1,\,1)}^2 + \|v\|_{L^2(-1,\,1)}^2} ,
 \qquad  \forall\;  U=(u,v)\in \cal{H}.
$$
Define an unbounded operator ${\cal A}: D({\cal A}) \subset
{\cal{H}}\rightarrow {\cal{H}}$ by
$$
{\cal A} U= \left( \begin{array}{c} v \\ (u'+ b v')' \end{array} \right) \quad \forall\; U
=\left( \begin{array}{c} u \\ v \end{array} \right) \in D({\cal A}),
$$
and
$$
  {\cal D}({\cal A})  =
 \Big\{(u,v)\in {\cal{H}} \: | \:
 v \in H^1_0(-1,\,1), \quad
(u'+ b v')' \in L^2(-1,\,1)\: \Big\}.
$$

\noindent Then system (\ref{system}) can be written as
\begin{equation}
\label{system1}
\frac{dU}{dt}  = {\cal A}U, \quad \forall \;t> 0, \quad  U(0)=\big(u_0 , v_0 \big)^T.
\end{equation}
It is easy to check that $\mathcal{A}$ is dissipative. Indeed,
a direct calculation yields
\begin{equation}\label{dissi}
\Re \langle \mathcal{A} U, U \rangle_{\mathcal{H}}=-\int_0^1 x^\a|v'|^2dx\leq 0.
\end{equation}

 If the coefficient
function $b(\cdot)\ge 0$ satisfies assumption \eqref{h1}, it is known (\cite{CLL}) that the following result on well-posedness of the system (\ref{system1}) hold by employing semigroup theories (see \cite{pazy}).
\begin{lemma}\label{l-2-1}
 $\mathcal{A}$  generates a contractive
$C_0$-semigroup on $\mathcal{H}$  and
\begin{equation}\label{eq}
i\mathbb{R} \subset  \rho (\mathcal{A}).
\end{equation}
\end{lemma}
We also have the following lemma which is from \cite{hanwang,LZ} deduced by the weighted Hardy's inequality (see \cite{stepanov}).

	\begin{lemma}\label{l-2-2}
  Assume that $\beta>-1$ and $\alpha<1$ are two constants. Then, there is $C=C(\alpha,\beta)>0$ so that
	\begin{eqnarray}\label{m2mm}
	 \int_{0}^{1}x^{\beta}|\xi(x)|^{2}dx \leq C \int_{0}^{1}x^{\alpha}|\xi_x(x)|^{2}dx,
	\end{eqnarray}
	for any $\xi(x)\in W^{1,1}(0,1)$ satisfying that $ x^{\frac{\alpha}{2}}\xi_x(x) \in L^{2}(0,1)$ and $\xi(1)=0$.
\end{lemma}

Our subsequent findings on polynomial stability will reply on the following result from  \cite{BT}, which gives necessary and sufficient conditions for a semigroup to be polynomially  stable.

\begin{lemma}\label{lemma-bt}
Let $A\,:\,D(A) \subset H \to H$ generate a bounded $C_0$-semigroup $e^{tA}$ on   Hilbert space  $H$.
Assume that
\begin{equation}
    \label{lem-spec}
 i\,\mathbb{R}\subset \rho(A).
   \end{equation}
Then the semigroup $e^{tA}$ is  polynomially stable of order $\displaystyle {1\over\theta} $ if and only if
   \begin{equation}
   \label{lem-poly}
\varlimsup\limits_{\omega\in \mathbb{R},|\omega|\to \infty}|\omega|^{- \theta}  \big\|(i\,\omega I-A)^{-1}\big\|_{{\cal L}(H)} < \infty.
 \end{equation}
 \end{lemma}

Our main result in this paper is the following.
\begin{theorem} \label{th-main}
	Assume that the damping coefficient $b(x)$ in (\ref{system}) satisfies the condition $(H1)$. Then system (\ref{system}) is polynomially stable of order $\frac{1-\a}{2-\a}$, i.e.,
		\begin{equation}\label{decay}
		\|e^{\mathcal{A}t}U_0\|_{\mathcal{H}}\leq
		C{t^{-{{2-\alpha}\over{1-\alpha}}}}\|U_0\|_{\mathcal{D}(\mathcal{A})}, \quad \forall \; U_0\in\mathcal{D}(\mathcal{A}),  \qquad t\ge 1,
		\end{equation}
where the constant $C>0$ is independent of $U_0$.
\end{theorem}

\section{Estimate for the Resolvent Operator(Proof of Theorem \ref{th-main})}
\setcounter{equation}{0} \setcounter{theorem}{0}
\setcounter{lemma}{0} \setcounter{definition}{0}

In this section, we shall  prove Theorem \ref{th-main}.
Due to Lemma \ref{lemma-bt}, along with Lemma \ref{l-2-1}, this is equivalent to show that there exists constant $r>0$ such that
\begin{equation}\label{huang2}
\inf\limits_{\|U\|_{\cal H}=1,\: \o\in{\mathbb R} } \o^\theta\|\ii\o U-{\cal A}U\|_{\cal H}  \ge r
\end{equation}
for the parameter
\begin{equation}
\label{p-beta}
\theta = {1-\alpha\over 2-\alpha}.
\end{equation}

Suppose that (\ref{huang2}) fails. Then there exist a
sequence of real numbers $\omega_n $  with $\omega_n\to \infty$   (suppose $\omega_n>1$ without  losing the generality) and a sequence of
functions
$\{U_n\}_{n=1}^\infty  = \{(u_n, v_{n})\}_{n=1}^\infty\subset D(\cA)$ with $\|U_n\|_{\cH}=1$ such that
\begin{equation}
\label{ex}
\omega_n^\theta\|\ii\omega_n U_n- \cA U_n\|_{\cH} =o(1),
\end{equation}
i.e.,
\begin{align}
\label{ex01}
& \omega_n^\theta\big(\mbox{i} \omega_n u_{n} -  v_n\big)
=o(1) &&\mbox{in}\; \;\;H^1(-1,\,1),\\ \noalign{\medskip}
\label{ex02}
&  \omega_n^\theta\big( \mbox{i} \omega_n v_{n} -  (u_n'+ x^\a v_n')'\big)
=o(1)&&\mbox{in}\;\;\;L^2(-1,\,1).
\end{align}

\noindent
Let
$$
u_{1,n} \doteq u_n\chi_{[-1, 0]},\quad v_{1,n} \doteq v_n\chi_{[-1, 0]},
$$
and
$$
u_{2,n} \doteq u_n\chi_{[0,1]},\quad v_{2,n} \doteq v_n\chi_{[0,1]},  \quad  T_n = u'_{2,n} +  x^\a v'_{2,n}.
$$
Then, by (\ref{ex}), one has
\begin{align}
\label{ex3}
& g_{1,n}\doteq \omega_n^\theta\big( \mbox{i} \omega_n u_{1,n} -  v_{1,n}\big)
=o(1)&&\mbox{in}\;\;\;H^1(-1,0),\\ \noalign{\medskip}
\label{ex4}
& g_{2,n}\doteq \omega_n^\theta\big( \mbox{i} \omega_n v_{1,n} -  u''_{1,n}\big)
=o(1) &&\mbox{in} \;\;\;L^2(-1,0),\\\noalign{\medskip}
\label{ex1}
& f_{1,n}\doteq \omega_n^\theta\big( \mbox{i} \omega_n u_{2,n} -  v_{2,n}\big)
=o(1) &&\mbox{in}\;\;\;H^1(0,1),\\ \noalign{\medskip}
\label{ex2}
& f_{2,n}\doteq  \omega_n^\theta\big(\mbox{i} \omega_n v_{2,n} -  T'_n\big)
=o(1) &&\mbox{in}\; \;\;L^2(0,1),
\end{align}
and the following connecting boundary conditions:
\begin{align}\label{b1}
T_n(0) = u_{1,n}'(0), \quad    u_{1,n}(0) = u_{2,n}(0), \quad  v_{1,n}(0) = v_{2,n}(0).
\end{align}

By (\ref{dissi}) and (\ref{ex}), we obtain that
\begin{equation}\label{311}
\Re\langle \omega_n^\theta(\ii\omega_n U_n- \cA U_n), U_n\rangle_{\mathcal{H}}=\omega_n^\theta \|  x^{\frac{\a}{2}}  v_{2,n}'\|_{L^2(0,1)}^2=o(1),
\end{equation}
which along with (\ref{ex1}) yields
\begin{equation}\label{312}
\|  x^{\frac{\a}{2} } u_{2,n}'\|_{L^2(0,1)}=\omega_n^{-1-\frac{\theta}{2}}o(1).
\end{equation}
\medskip
In what follows, we shall  reach a contradiction by showing $\|U_n\|_{\cal H}=o(1)$.
\begin{lemma} For $n\to \infty,$ one has the following estimates
	\begin{equation}
	| v_{2,n}(0)|= \omega_n^{-{\theta\over2}}o(1),\quad \|v_{2,n}\|_{L^2(0,1)}=\omega_n^{-{\theta\over2}}o(1).
	\label{v-uni}
	\end{equation}
\end{lemma}

\begin{proof}
	A direct computation gives that when $0\le\alpha <1$,
	\begin{equation}\label{314}
	|v_{2,n}(x)|  = \Big|\int_x^1 v_{2,n}'dx \Big| \le  \Big(\int_x^1 x^\alpha | v_{2,n}'|^2 dx \Big)^{1\over2}
	\Big(\int_x^1 x^{-\alpha} dx \Big)^{1\over2}.
	\end{equation}
	Combining this with \eqref{311} and letting $x=0$, we obtain the first estimate in  \eqref{v-uni}.
	
	Moreover, by  (\ref{314}),  along with (\ref{311}), we have
	\begin{equation}
	\int_0^1 |v_{2,n}| ^2dx\leq \int_0^1 \Big(\int_x^1 x^\alpha | v_{2,n}'|^2 dx \Big)
	\Big(\int_x^1 x^{-\alpha} dx \Big)dx=\omega_n^{-\theta}o(1).
	\end{equation}
	Thus, the second estimate in (\ref{v-uni}) follows.
\end{proof}

\vskip 4mm

\begin{lemma} \label{L-3-2}
	
	Set $\xi_n \doteq  \omega_n^{-\delta}$ and
	$Q_n \doteq [\xi_n\slash 2,\;  \xi_n]$, where
	$\delta=\frac{1}{2-\alpha}$. One has that
	\begin{equation}
	\label{txi}
	|T_n(\xi_n)| =o(1).
	\end{equation}
\end{lemma}

\begin{proof}
	By (\ref{311}) and (\ref{312}),  a direct computation gives
	\begin{equation}\label{eb1}
	\begin{array}{lcl}
	\min\limits_{x\in Q_n}
	| T_n(x)|  &\le& \displaystyle \Big[4\omega_n^{\delta}\int_{Q_n} (|u_{2,n}'|^2+|av_{2,n}'|^2)dx\Big]^{1\over2}
	\\ \noalign{\medskip} &\le & \displaystyle
	2\omega_n^{\delta\over2}\Big(
	\max\limits_{x\in Q_n}|x^{-\alpha}|\,\int_{Q_n} x^{\alpha} |u_{2,n}'|^2dx+ \max\limits_{x\in Q_n}|x^{\alpha}|\,\int_{Q_n} x^{\alpha} |v_{2,n}'|^2dx\Big)^{1\over2}
	\\ \noalign{\medskip}  &\le & \displaystyle
	\omega_n^{{1\over2} (\delta+\delta\alpha -2 - \theta )}o(1)
	+\omega_n^{{1\over2} ( \delta-\delta\alpha   - \theta )}o(1)=o(1).
	\end{array}
	\end{equation}
	Here we have used  $\delta=\frac{1}{2-\alpha}$ and   $\theta=\frac{1-\alpha}{2-\alpha}$ as given in (\ref{p-beta}).
	Hence, (\ref{txi}) holds.
\end{proof}

\vskip 4mm

Integrating \eqref{ex2} on  $(0,\, \xi_n)$ yields
\begin{equation}
\label{y2int2}
\ii \omega_n \int_0^{\xi_n} v_{2,n} dx +T_n(0)- T_n(\xi_n) = \omega_n^{-\theta} o(1).
\end{equation}

\noindent
 We now claim that
\begin{equation}
\label{eb41}
\Big|\omega_n \int_0^{\xi_n} v_{2,n} dx\Big|=o(1),
\end{equation}
 and prove it later.

 With (\ref{eb41}) at hand, let's substitute \eqref{txi} and \eqref{eb41} into \eqref{y2int2} to get
\begin{equation} \label{320}
|T_n(0)| =o(1).
\end{equation}

\noindent
Then, multiplying \eqref{ex2} with $\omega_n^{-\theta}u_{2,n}$ and integrating it on $(0,1)$,  along with (\ref{ex1}), (\ref{311}) and  (\ref{320}),
we obtain
\begin{equation}  \label{equa}
\|u_{2,n}' \|^2_{L^2(0,1)}  -  \|v_{2,n}\|^2_{L^2(0,1)} =o(1),
\end{equation}
which together with  the second estimate in (\ref{v-uni}) implies that
\begin{equation}\label{322}
\|u_{2,n}' \|^2_{L^2(0,1)}=o(1).
\end{equation}

\noindent
Taking $L^2(-1,0)-$inner product of (\ref{ex2})  with
$\omega_n^{-\theta}(x+1)u_{1,n}'$, and then integrating by parts, we have
\begin{equation}\label{ss31}
-|v_{1,n}(0)|^{2}+\|v_{1,n}\|^2_{L^2(-1,0)}-|u_{1,n}'(0)|^{2}+\|u_{1,n}'\|^2_{L^2(-1,0)}= o(1).
\end{equation}

\noindent
Thus, by (\ref{320}), the first estimate in  (\ref{v-uni}) and the transmission conditions  (\ref{b1}) at the interface, we
get
\begin{equation}\label{ss32}
\|v_{1,n}\|^2_{L^2(-1,0)}+\|u_{1,n}'\|^2_{L^2(-1,0)}=o(1).
\end{equation}

\noindent
Finally, by the first estimate in (\ref{v-uni}), (\ref{322}) and (\ref{ss32}), we have achieved the contradiction:
\begin{equation}\label{ss30}
\|U_{n}\|_\mathcal{H}=o(1).
\end{equation}

Now, in order to     complete the proof, it is sufficient to show (\ref{eb41}) holds.

Indeed, by the H\"{o}lder inequality, we have
\begin{eqnarray}
\Big|\omega_n \int_0^{\xi_n} v_{2,n} dx\Big|&\leq&  C \omega_n^{1-\frac{\delta}{2}}\Big(\int_0^{\xi_n}|v_{2,n}|^2dx\Big)^{\frac{1}{2}}\cr
&\leq& \displaystyle C \omega_n^{1-\frac{\delta}{2}} \Big(\Re \int_0^1x v_{2,n}\overline{v_{2,n}'}dx\Big)^{\frac{1}{2}}\cr
&\leq& \displaystyle C \omega_n^{1-\frac{\delta}{2}}  \Big(\int_0^1x^\alpha |v_{2,n}'|^2dx \Big)^{\frac{1}{4}} \Big(\int_0^1 x^{2-\alpha} |v_{2,n}|^2dx \Big)^{\frac{1}{4}}.
\end{eqnarray}
If we can show  the estimate (\ref{316+}) in the following Lemma \ref{l3-3} holds, then above inequality leads to
\begin{equation}
\Big|\omega_n \int_0^{\xi_n} v_{2,n} dx\Big|
\leq\omega_n^{1-\frac{\delta}{2}} \omega_n^{-\frac{\theta}{4}}\omega_n^{\frac{-2-\theta}{4}}o(1)=o(1)\quad ({\rm using\;\; (\ref{311})\; and\; (\ref{316+})})
\end{equation}
 since
$$ 1-\frac{\delta}{2} -\frac{\theta}{4} - \frac{2+\theta}{4} = \frac{1}{2}-\frac{1}{2(2-\a)} - \frac{1-\a}{2(2-\a)} = 0.
$$

\noindent
Hence, (\ref{eb41}) holds.

\vskip 4mm
The rest is devoted to showing the following Lemma \ref{l3-3}.
\begin{lemma}\label{l3-3}
	Let $0 \leq\alpha<1$. Then it holds that	
	\begin{equation}\label{316+}
	\omega_n\int_0^1 x^{2-\alpha} |v_{2,n}|^2dx=\omega_n^{-1-\theta} o(1).
	\end{equation}
\end{lemma}
\begin{proof}
	Multiplying \eqref{ex2} with $x^{2-\alpha} \overline{v_{2,n}}$ and integrating it from $0$ to $1$, we get
	$$
	\begin{array}{ll}
	&\displaystyle
	\ii \omega_n \int_0^1 x^{2-\alpha}|v_{2,n}|^2dx + \int_0^1T_n\, \overline{((2-\alpha)x^{1-\alpha}  v_{2,n}+ x^{2-\alpha} v_{2,n}')}dx
	\\ \noalign{\medskip}  =& \displaystyle
	\ii \omega_n \int_0^1 x^{2-\alpha}|v_{2,n}|^2dx +\int_0^1 \big[(2-\alpha) x^{1-\alpha}  u_{2,n}' \overline{v_{2,n}}+ x^{2-\alpha} u_{2,n}' \overline{v_{2,n}'}
\\ \noalign{\medskip}  &\displaystyle
+ (2-\alpha) x v_{2,n}'\overline{ v_{2,n}}+ x^{\alpha +2-\alpha} |v_{2,n}'|^2\big]dx
	\\ \noalign{\medskip}  =& \displaystyle \omega_n^{-\theta} \int_0^1x^{2-\alpha} f_{2,n}\overline{v_{2,n}}dx .
	\end{array}
	$$
	By taking the imaginary parts of the above equality and using Cauchy-Schwarz inequality, we obtain
	\begin{equation}
	\label{multiply}
\begin{array}{ll}& \displaystyle
\omega_n \int_0^1 x^{2-\alpha}|v_{2,n}|^2dx
 \\ \noalign{\medskip}  \displaystyle
	\le&\displaystyle   2\Big|
	\int_0^1 [(2-\alpha)x^{1-\alpha}  u_{2,n}' \overline{v_{2,n}}\!+\! x^{2-\alpha} u_{2,n}' \overline{v_{2,n}'}\!+\! (2-\alpha) x v_{2,n}'\overline{ v_{2,n}}  ]dx\Big|+ \omega_n^{-2\theta-1} o(1).
	\end{array}
\end{equation}
	
	In the following,  in order to show (\ref{316+}),  we shall estimate the three terms on the right hand side of (\ref{multiply}), respectively.
	
	\noindent{\bf Observation I.}
	By H\"{o}lder inequality, we have
	\begin{equation}\label{338}
	\Big|\int_0^1   x^{1-\alpha}  u_{2,n}' \overline{v_{2,n}}dx\Big|\le
	\Big(\int_0^1 x^\alpha |u_{2,n}'|^2dx\Big)^{\frac{1}{2}} \Big(\int_0^1 x^{2-3\alpha} |v_{2,n}|^2dx\Big)^{\frac{1}{2}}.
	\end{equation}
	Note that $2-3\alpha>-1$. Then by Lemma \ref{l-2-2} and \eqref{311}, one has that
	\begin{equation}\label{339}
	\Big(\int_0^1 x^{2-3\alpha} |v_{2,n}|^2dx\Big)^{\frac{1}{2}} \le
	C\| x^{\alpha \over2} v_{2,n}' \|_{L^2(0,1)} = \omega_n^{-{
		\theta\over2}} o(1).
	\end{equation}
	Then, substituting (\ref{312}), (\ref{339}) into (\ref{338}),  we obtain
	\begin{equation}\label{3381}
	\Big|\int_0^1   x^{1-\alpha}  u_{2,n}' \overline{v_{2,n}}dx\Big|  \leq \omega_n^{-1-\theta} o(1).
	\end{equation}
	
	\noindent	{\bf Observation II.} We have
	\begin{equation}
	\Big|\int_0^1 x^{2-\alpha} u_{2,n}' \overline{v_{2,n}'}dx\Big|\leq \omega_n^{-1-\theta}o(1).
	\end{equation}
	
	\noindent
	In fact,  by the H\"{o}lder inequality, we have
	\begin{eqnarray*}
	\Big|\int_0^1  x^{2-\alpha} u_{2,n}' \overline{v_{2,n}'} dx
	\Big|&=&\Big|\int_0^1  x^{2-2\alpha} x^{\frac{\alpha}{2}}u_{2,n}'x^{\frac{\alpha}{2}} \overline{v_{2,n}'} dx
	\Big|
\\ \noalign{\medskip}  \displaystyle
	&\le&  \| x^{\alpha \over 2}  u_{2,n}'\|_{L^2(0,1)}\| x^{\alpha \over 2}  v_{2,n}'\|_{L^2(0,1)} = \omega_n^{-1-\theta}o(1).
	\end{eqnarray*}
	Here we have used the fact that $0\leq \alpha<1$.
	
	\noindent	{\bf Observation III.}  We  have
	\begin{equation}
	\Big|\int_0^1	x v_{2,n}'\overline{ v_{2,n}}dx\Big| \leq \varepsilon \omega_n\| x^{2-\alpha \over2} v_{2,n} \|_{L^2(0,1)}^2
	+  \omega_n^{-1-  \theta  } o(1),
	\end{equation}
	where $\varepsilon$ is some positive constant which can be  chosen arbitrarily. In fact, by H\"{o}lder inequality, for a positive constant $\varepsilon$, there exists $C_\varepsilon$ such that
	$$
	\begin{array}{l}\displaystyle \hskip 5mm
	\Big|\int_0^1   x v_{2,n}'\overline{ v_{2,n}}  dx\Big|
	\\ \noalign{\medskip}  \displaystyle
	\le
	\|   x^{\alpha \over2}   v_{2,n}'\|_{L^2(0,1)}
	\|   x^{2-\alpha \over2}   v_{2,n}\|_{L^2(0,1)}
	\\ \noalign{\medskip}  \displaystyle
	\le  \varepsilon \omega_n\| x^{2-\alpha \over2}  v_{2,n} \|_{L^2(0,1)}^2
	+ C_\varepsilon \omega_n^{-1} \|   x^{ \alpha \over2}   v_{2,n} '\|_{L^2(0,1)} ^2
	\\ \noalign{\medskip}  \displaystyle
	=   \varepsilon \omega_n\| x^{2-\alpha \over2} v_{2,n} \|_{L^2(0,1)}^2
	+  \omega_n^{-1-  \theta  } o(1).
	\end{array}$$
	
	\noindent
	By substituting {\bf Observation I, II, III} into the right hand side of (\ref{multiply}), we obtained the desired result (\ref{316+}).
\end{proof}

\section{Conclusion}
In this paper, we obtain a sharper estimate of the decay rate of solution to
system (\ref{system}) when the damping coefficient function $a(x)$ is equivalent to $x^\a, 0\le\a<1$ near the interface $x=0$, which is consistent with the existing optimal decay rate when $\a=0$, and with the known exponential decay rate when  $\a=1$.  We summarize the stability results for (\ref{system}) in the following table.

\begin{center}
	\begin{tabular}{| c | c | c | c |}
		\hline
		$\a$ &  damping coefficient $a(x)$ & decay rate of solution \\ \hline
		$0$ &  $x^\alpha$ & optimal polynomial decay rate $t^{-2}$ \\ \hline
		$(0,1)$  & $x^\a$ &polynomial decay rate $t^{-\frac{2-\a}{1-\a}}$ \\ \hline
		$\ge 1$ &$x^\alpha$ & exponential decay rate \\ \hline				
	\end{tabular}
\end{center}

We would like to point out that the above results are proved only for the damped region which is an interval including one end point of the spatial domain. It remains as an open question for the general case where the damping coefficient function $b(x)$ is supported on a proper interval $(x_1, x_2)$ with $-1<x_1<x_2<1$, and behaves like $(x-x_1)^\a$ and
$(x-x_2)^\beta$ near the interfaces $x=x_1, x_2$ respectively. The optimality of the polynomial decay rate obtained in this paper is another open question
as stated in the Introduction.

\end{document}